\documentclass[]{article}
\usepackage{times}
\usepackage{amssymb,amsthm,latexsym,amsmath,epsfig,pgf}
\usepackage{mathtools}
\usepackage{hyperref}

\usepackage{enumerate} 

\usepackage{algorithmic_ok_sok}
\usepackage[section]{algorithm}
\makeatletter
\floatname{algorithm}{Algorithm}
\makeatother

\usepackage{tikz}
\usetikzlibrary{backgrounds}
\usetikzlibrary{decorations.pathmorphing,shapes}
\usetikzlibrary{arrows.meta}

\usetikzlibrary{arrows,positioning,intersections}
\usetikzlibrary{backgrounds}
\usetikzlibrary{trees}

\newtheorem{definition}{Definition}[section]
\newtheorem{observation}[definition]{Observation}

\newtheorem{theorem}[definition]{Theorem}
\newtheorem{corollary}[definition]{Corollary}
\newtheorem{lemma}[definition]{Lemma}

\newtheorem*{noCounterThm}{Theorem}

\newtheorem{property}{Property}

\newcommand{\tw}{\operatorname{tw}}

\newcommand{\veident}{\mathbin{\begin{tikzpicture}[scale = 1, baseline = -0.3ex]
		\draw [line width=.45mm, black] (-.075,.13) -- (.075,.13);
		\draw [black, fill = black, radius = 2pt] (0,0) circle;
		\end{tikzpicture}}}
\newcommand{\vident}{\mathbin{\begin{tikzpicture}[scale = 1, baseline = -0.3ex]
		\draw [black, fill = black, radius = 2pt] (0,.05) circle;
		\end{tikzpicture}}}
\newcommand{\eident}{\mathbin{\begin{tikzpicture}[scale = 1, baseline = -0.3ex]
		\draw [line width=.45mm, black] (-.075,0) -- (.075,0);
		\draw [line width=.45mm, black] (-.075,.13) -- (.075,.13);
		\end{tikzpicture}}}

\title{Cycle Decompositions and Constructive Characterizations}
\author{Irene Heinrich and Manuel Streicher}

\begin{document}

\maketitle

\begin{abstract}
	Decomposing an Eulerian graph into a minimum respectively maximum number of edge disjoint cycles is an NP-complete problem. We prove that an Eulerian graph decomposes into a unique number of cycles if and only if it does not contain two edge disjoint cycles sharing three or more vertices.
	To this end, we discuss the interplay of three binary graph operators leading to novel constructive characterizations of two subclasses of Eulerian graphs.
	This enables us to present a polynomial-time algorithm which decides whether the number of cycles in a cycle decomposition of a given Eulerian graph is unique.
\end{abstract}

\section{Introduction}
It is well-known that a graph is Eulerian if and only if its edge set can be decomposed into cycles (cf. \cite{Fleischner90}).
The decision problem whether an Eulerian graph can be decomposed into at most $k$ cycles is NP-complete as a consequence of \cite{Peroche}.
Also the corresponding maximization problem is NP-complete, cf. \cite{Holyer81}.

Our contribution is to give two equivalent characterizations for the class of Eulerian graphs where both numbers -- the minimum and the maximum amount of cycles in a cycle decomposition -- coincide. We show that those are exactly the graphs that can be constructed from the set of Eulerian multiedges using a finite number of  vertex-identifications and vertex-edge-identifications which will be introduced and discussed in Section \ref{sec: construction operations}. This constructive characterization then enables us to prove the following statement:

\begin{noCounterThm}[\ref{thm: main} - \emph{shortened version}]
	Let $G$ be an Eulerian graph. The number of cycles in a cycle decomposition of $G$ is unique if and only if no two edge-disjoint cycles in $G$ intersect more than twice.
\end{noCounterThm}

We exploit Theorem \ref{thm: main} to develop an algorithm which applies the identification operations backwards.
We can recognize the described graph class in polynomial time.

\begin{noCounterThm}[\ref{thm:cycleunique} - \emph{shortened version}]
	We can decide in time $\mathcal{O}(n(n+m))$ if the number of cycles in a cycle decomposition of a given Eulerian graph is unique.
\end{noCounterThm}

Our main tool for proving the before mentioned results is a novel \emph{constructive characterization}.
A constructive characterization of a graph class is a construction manual for building all graphs in the class starting from some simple set of initial graphs.
Many graph classes can be expressed through constructive characterization, among those are graphs of low  treewidth \cite{Bodlaender98}, $3$-connected and $k$-edge-connected graphs \cite{Frank}.

We may turn the before mentioned statement -- \emph{a graph is Eulerian if and only if it is connected and can be decomposed into cycles} -- into a toy example for a constructive characterization.
We describe the class $\mathcal{E}$ of Eulerian graphs recursively:
\begin{itemize}
	\item If $G$ is isomorphic to $K_1$ or $C_n$ for some $n \in \mathbb{N}, n \geq 1$ then $G \in \mathcal{E}$
	\item If $G_1, G_2 \in \mathcal{E}$ with $E(G_1) \cap E(G_2) = \emptyset$ and $V(G_1) \cap V(G_2) \neq \emptyset$, then also $G_1 \cup G_2 \in \mathcal{E}$.
\end{itemize}
Often constructive characterizations can be exploited to prove a desired statement by induction.
Coming back to the above toy example, we can prove that every Eulerian graph has only vertices of even degree by first observing that each $C_n$ and the $K_1$ have only vertices of even degree and then using the fact that the graph union with disjoint edge sets does not change the even degrees.

We study three basic binary graph operators.
In Section~\ref{sec: construction operations} we define these operators and regard their  behaviour concerning the following graph invariants: connectivity, minimum and maximum number of cycles in a cycle decomposition and treewidth.
Sections~\ref{sec: subquartic} and \ref{sec: nu = c} will then use the introduced operators for constructive characterizations of
Eulerian graphs with maximum degree at most 4 and treewidth at most 2 (Section \ref{sec: subquartic}) and
Eulerian graphs which have the property that the number of cycles in all of its cycle decompositions is the same (Section \ref{sec: nu = c}).
Finally, in Section \ref{sec: algo} we exploit the gained insights to develop a polynomial time algorithm which decides if the cycle number of a given Eulerian graph is unique.

\section{Preliminaries}
We use standard graph terminology, see~\cite{West01, Diestel00, Krumke09}.
Though, we recall some basic notions in the following.
A \emph{graph} $G$ is a triple consisting of a finite non-empty \emph{vertex set} $V(G)$, and a finite \emph{edge set} $E(G)$ and a relation that associates with each edge two different vertices called its \emph{end vertices}.
Observe that this definition excludes loop edges.
If two edges have the same two endvertices we call them \emph{parallel}.
An edge with end vertices $u$ and $v$ is often written as $uv$.
We use this notation even if $G$ has parallel edges between the vertices $u$ and $v$.
This does not lead to any inconvenience as the problems discussed in this article refer to graph invariants which do not depend on the choice of the exact edge between $u$ and $v$.
We denote with $N_G(u)$ the set of all neighbours of $u$ in $G$.
The \emph{degree} $\deg_G(v)$ of $v \in V(G)$ is defined as the number of edges incident to $v$.
If all vertices of $G$ have the same degree $k$, then $G$ is $k$-regular.
We call two graphs~$G$ and $G'$ disjoint if $V(G) \cap V(G') = \emptyset$ and $E(G) \cap E(G') = \emptyset$.
Let $G$ and $G'$ be two graphs with $E(G) \cap E(G') = \emptyset$.
We set $G \cup G'$ to be the graph with $V(G \cup G') = V(G) \cup V(G')$ and $E(G \cup G') = E(G) \cup E(G')$.

Let $u \in V(G)$. We denote with $G-u$ the graph where $u$ and all its incident edges are removed from $G$.
For $F \subseteq E(G)$ we write $G-F$ for the graph with $V(G-F) = V(G)$ and $E(G-F) = E(G)\setminus F$.
If $F = \{f\}$ we write~${G-f}$.
Let $G$ be a graph containing a vertex~$u \in V(G)$ with $\deg_G(u) = 2$ with two distinct neighbours.
\emph{Resolving} $u$ means to remove~$u$ from $G$ and to connect its two neighbours by a new edge.
A \emph{path} $P$ is a graph of the form $V(P) = \{u_0, u_1, \dots, u_k \}$, $E(P) = \{u_0u_1, u_1u_2, \dots, u_{k-1}u_k\}$, where all the $u_i$ are distinct.
We often refer to a path omitting its precise edges but only listing the sequence of its vertices ordered according to their appearance in $P$, say $P = u_0u_1\dots u_k$.
We say that $P$ is a $u_0$-$u_k$-path, the vertices $u_1, \dots, u_{k-1}$ are \emph{internal}\index{internal vertex} vertices of $P$.
Let $P$ be a $u$-$v$-path and $Q$ be a $v$-$w$-path with $V(P) \cap V(Q) = v$.
We set $PQ \coloneqq P \cup Q$.
Even when we study paths as subgraphs of non-simple graphs, this notation does not lead to any inconvenience: In the upcoming topics it is never of any relevance which precise edge a path uses.
If~${P = u_0 \dots u_k}$ is a path, then the graph $C \coloneqq P \cup u_ku_0$ is a \emph{cycle}.
A \emph{cycle decomposition} of a graph $G$ is a set of cycles which are subgraphs in $G$ such that each edge appears in exactly one cycle in the set.
We set
\begin{align*}
c(G) &\coloneqq \min\{|\mathcal{C}| \colon \mathcal{C} \text{ is a cycle decomposition of } G\},\\
\nu(G) &\coloneqq \max\{|\mathcal{C}| \colon \mathcal{C} \text{ is a cycle decomposition of } G\}
\end{align*}
to be the \emph{minimum} respectively \emph{maximum cycle number} of $G$.
A graph is \emph{Eulerian} if it allows for an \emph{Euler tour}, i.e.\ a non-empty alternating sequence $v_0e_0v_1e_1\dots e_{k-1}v_k$ of vertices and edges in $G$ such that $e_i$ has end vertices $v_i$ and $v_{i+1}$ for all $0 \leq i < k$, $v_0 = v_k$ and every edge of $G$ appears exactly once in the sequence.

A graph $G$ is called \emph{connected}\index{connected} if it is non-empty and any two of its vertices are linked by a path in $G$.
The \emph{components}\index{components} of a graph are its maximal (with respect to the subgraph relation) connected subgraphs.
For $V_1, V_2 \subseteq V(G)$ we set $E(V_1, V_2)$ to be the set of all edges with one endvertex in $V_1$ and the other endvertex in $V_2$.
A set $F$ of edges is a \emph{cut}\index{cut} in  $G$ if there exists a partition $\{V_1, V_2\}$ of $V$ such that ${F = E(V_1, V_2)}$.
We call~$F$ a $k$-cut if $|F| = k$.
An element of a~1-cut is called a \emph{cut-edge}.
A connected graph~$G$ is called $k$-edge-connected if it stays connected after the removal of $k-1$ arbitrary edges.
A vertex $v \in V(G)$ is a \emph{cut-vertex} if $G-v$ has more connected components than~$G$.
A connected graph without cut-vertices is called \emph{biconnected}.
The maximal biconnected subgraphs of a graph are called its \emph{biconnected components}.
For a more detailed description of biconnectivity and some basic results we refer to~\cite{West01} and~\cite{Hopcroft1973}.
We say that a set $S \subseteq V(G) \cup E(G)$ \emph{separates} $w_1, w_2 \in V(G)$ if there exists no $w_1$-$w_2$-path in $G$ without elements of $S$.

A connected graph $T$ that does not contain a cycle as a subgraph is a \emph{tree}.
A vertex of degree~1 in $T$ is called a leaf.
For a graph $G$ a \emph{tree-decomposition} $(\mathcal{T},\mathcal{B})$ of $G$ consists of a tree~$\mathcal{T}$ and a set~$\mathcal{B} = \{B_t \colon t \in V(\mathcal{T}) \}$
of \emph{bags} $B_t \subseteq V(G)$ such that
$V(G) = \bigcup\limits_{t \in V(\mathcal{T})} B_t,$
for each edge~${vw \in E(G)}$ there exists a vertex $t \in V(\mathcal{T})$ such that
$v,w \in B_t$, and
if $v \in B_{s} \cap B_{t},$ then~${v \in B_r}$ for each vertex
$r$  on the path connecting~$s$ and $t$ in $\mathcal{T}$.
A tree-decomposition $(\mathcal{T},\mathcal{B})$ has \emph{width} $k$ if each bag has a size of at most $k+1$ and there
exists some bag of size $k+1$.
The \emph{treewidth} of $G$ is the smallest integer~$k$ for which there is a width $k$ tree-decomposition of $G$.
We write $\tw(G) = k$.
A tree-decomposition $(T,\mathcal{B})$ of width $k$ is \emph{smooth} if
$ |B_t| = k+1$ for all~${t \in V(\mathcal{T})}$ and $|B_s \cap B_{t}| = k$ for all $st \in E(\mathcal{T})$.
A graph of treewidth at most $k$ always has a smooth tree-decomposition of width~$k$; see Bodlaender~\cite{Bodlaender98}.

The \emph{contraction}\index{contraction} of an edge $e$ with endpoints $u,v$ is the replacement of $u$ and $v$ with a single vertex whose incident edges are the edges other than $e$ that were incident to $u$ or $v$.
A graph $H$ is a \emph{minor}\index{minor} of a graph $G$ if an isomorphic copy of $H$ can by obtained from $G$ by deleting or contracting edges of $G$.
The graph $H$ obtained by \emph{subdivision} of some edge $uv\in E(G)$ is obtained by replacing the edge $uv$ by a new vertex $w$ and edges $uw$
and $wv$.

\section{Construction operations}
\label{sec: construction operations}
In the following, we introduce three binary graph operations -- vertex-identification, edge-identification and vertex-edge-i\-den\-ti\-fi\-ca\-tion.
The constructive characterizations in Section ~\ref{sec: subquartic} and~\ref{sec: nu = c} each start off by a simple base class of graphs.
In Section~\ref{sec: subquartic} the considered class is then built by mainly using edge-identification.
In Section \ref{sec: nu = c} the vertex-edge-identification is the crucial construction tool.
After defining the above mentioned operations we regard their behaviour concerning cycle decompositions, connectivity and treewidth.

\paragraph{Vertex-identification}
Let $G_1$ and $G_2$ be disjoint graphs and let $u_1 \in V(G_1)$, $u_2 \in V(G_2)$.
We construct the graph $(G_1, u_1) \vident (G_2, u_2)$ by identifying $u_1$ and $u_2$.

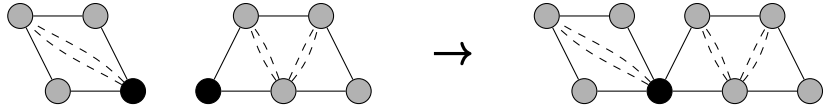
\begin{figure}[h]%
	\centering
	\begin{tikzpicture}[node distance=1cm]
	\node[circle, draw, fill=black] (1) {};
	\node[circle, draw, left of=1, fill=black!30] (2) {};
	\node[circle, draw, above of=2, xshift=-0.5cm, fill=black!30] (3) {};
	\node[circle, draw, right of=3, fill=black!30] (4) {};
	\node[circle, draw, right of=4, fill=black!30] (5) {};
	\node[circle, draw, right of=5, fill=black!30] (6) {};
	\node[circle, draw, below of=6, xshift=0.5cm, fill=black!30] (7) {};
	\node[circle, draw, left of=7, fill=black!30] (8) {};
	\draw[-, dashed] (1) to[out=135, in=-30] (3);
	\draw[-, dashed] (1) to[out=150, in=-45] (3);
	\path[draw] (1) -- (2) -- (3) -- (4);
	\draw[-] (4) -- (1);
	\draw[-] (5) -- (1);
	\path[draw] (5) -- (6) -- (7) -- (8) -- (1);
	\draw[-, dashed] (8) to[out=110, in=-55] (5);
	\draw[-, dashed] (8) to[out=125, in=-70] (5);
	\draw[-, dashed] (8) to[out=50, in=-105] (6);
	\draw[-, dashed] (8) to[out=65, in=-120] (6);
	
	\node[circle, draw,fill=black, left of=1, xshift=-6cm] (10) {};
	\node[circle, draw,fill=black, right of=10] (11) {};
	\node[circle, draw, left of=10, fill=black!30] (12) {};
	\node[circle, draw, above of=12, xshift=-0.5cm, fill=black!30] (13) {};
	\node[circle, draw, right of=13, fill=black!30] (14) {};
	\node[circle, draw, right of=14, xshift=1cm, fill=black!30] (15) {};
	\node[circle, draw, right of=15, fill=black!30] (16) {};
	\node[circle, draw, below of=16, xshift=0.5cm, fill=black!30] (17) {};
	\node[circle, draw, left of=17, fill=black!30] (18) {};
	\draw[-, dashed] (10) to[out=135, in=-30] (13);
	\draw[-, dashed] (10) to[out=150, in=-45] (13);
	\path[draw] (10) -- (12) -- (13) -- (14);
	\path[draw] (15) -- (16) -- (17) -- (18) -- (11);
	\draw[-, dashed] (18) to[out=110, in=-55] (15);
	\draw[-, dashed] (18) to[out=125, in=-70] (15);
	\draw[-, dashed] (18) to[out=50, in=-105] (16);
	\draw[-, dashed] (18) to[out=65, in=-120] (16);
	\draw[-] (10) to (14);
	\draw[-] (11) to (15);
	
	\coordinate[left of=1, xshift=-2cm, yshift=0.5cm] (100);
	\coordinate[right of=100, xshift=-0.5cm] (101);
	\draw[->, very thick] (100) to (101);

	\end{tikzpicture}
	\caption{Vertex-identification of two Eulerian graphs}%
	\label{fig:vertex-identification}%
\end{figure}

\paragraph{Edge-identification}
Let $G_1, G_2$ be disjoint graphs. Further let $e_i \in E(G_i)$ be an edge with endpoints $u_i, v_i$ for $i \in \{1,2\}$.
We construct $(G_1,e_1,u_1) \eident (G_2, e_2,u_2)$ by removing the edge~$e_i$ from $G_i$, $i \in \{1,2\}$ and adding an edge from $u_1$ to $u_2$ and another one from $v_1$ to $v_2$.

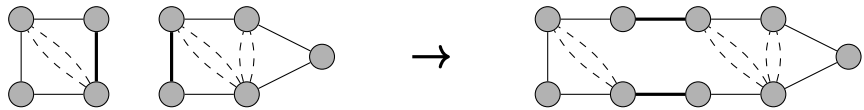
\begin{figure}[h]%
	\centering
	\begin{tikzpicture}[node distance=1cm]
	\node[circle, draw, fill=black!30] (1) {};
	\node[circle, draw, right of=1, fill=black!30] (1b) {};
	\node[circle, draw, left of=1, fill=black!30] (2) {};
	\node[circle, draw, above of=2, fill=black!30] (3) {};
	\node[circle, draw, right of=3, fill=black!30] (4) {};
	\node[circle, draw, right of=4, fill=black!30] (5) {};
	\node[circle, draw, right of=5, fill=black!30] (6) {};
	\node[circle, draw, below of=6, xshift=1cm, yshift=0.5cm, fill=black!30] (7) {};
	\node[circle, draw, left of=7, yshift=-0.5cm, fill=black!30] (8) {};
	\draw[-, dashed] (1) to[out=120, in=-35] (3);
	\draw[-, dashed] (1) to[out=145, in=-60] (3);
	\path[draw] (1) -- (2) -- (3) -- (4);
	\draw[-, very thick] (4) -- (5);
	\path[draw] (5) -- (6) -- (7) -- (8) -- (1b);
	\draw[-, very thick] (1b) -- (1);
	\draw[-, dashed] (8) to[out=120, in=-35] (5);
	\draw[-, dashed] (8) to[out=145, in=-60] (5);
	\draw[-, dashed] (8) to[out=75, in=-75] (6);
	\draw[-, dashed] (8) to[out=105, in=-105] (6);
	
	\coordinate[left of=1, xshift=-1.8cm, yshift=0.5cm] (100);
	\coordinate[right of=100, xshift=-0.5cm] (101);
	\draw[->, very thick] (100) to (101);
	
	\node[circle, draw, left of=1, xshift=-6cm, fill=black!30] (10) {};
	\node[circle, draw, right of=10, fill=black!30] (1b0) {};
	\node[circle, draw, left of=10, fill=black!30] (20) {};
	\node[circle, draw, above of=20, fill=black!30] (30) {};
	\node[circle, draw, right of=30, fill=black!30] (40) {};
	\node[circle, draw, right of=40, fill=black!30] (50) {};
	\node[circle, draw, right of=50, fill=black!30] (60) {};
	\node[circle, draw, below of=60, xshift=1cm, yshift=0.5cm, fill=black!30] (70) {};
	\node[circle, draw, left of=70, yshift=-0.5cm, fill=black!30] (80) {};
	\draw[-, dashed] (10) to[out=120, in=-35] (30);
	\draw[-, dashed] (10) to[out=145, in=-60] (30);
	\path[draw] (10) -- (20) -- (30) -- (40);
	\draw[-, very thick] (10) -- (40);
	\path[draw] (50) -- (60) -- (70) -- (80) -- (1b0);
	\draw[-, very thick] (1b0) -- (50);
	\draw[-, dashed] (80) to[out=120, in=-35] (50);
	\draw[-, dashed] (80) to[out=145, in=-60] (50);
	\draw[-, dashed] (80) to[out=75, in=-75] (60);
	\draw[-, dashed] (80) to[out=105, in=-105] (60);

	\end{tikzpicture}
	\caption{Edge-identification of two Eulerian graphs}%
	\label{fig:Edge-identification}%
\end{figure}

\paragraph{Vertex-edge-identification}
Let $G_1$ and $G_2$ be disjoint graphs and let $e_i$ be an edge in $G_i$ from $u_i$ to $v_i$ for $i \in \{1,2\}$.
We define $(G_1, e_1, u_1) \veident (G_2, e_2, u_2)$ to be the graph where $v_1$ and $v_2$ are identified, the edges $e_1, e_2$ are removed and an edge between $u_1$ and $u_2$ is added.

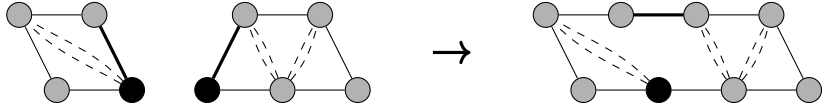
\begin{figure}[h]%
	\centering
	\begin{tikzpicture}[node distance=1cm]
	\node[circle, draw, fill=black] (1) {};
	\node[circle, draw, left of=1, fill=black!30] (2) {};
	\node[circle, draw, above of=2, xshift=-0.5cm, fill=black!30] (3) {};
	\node[circle, draw, right of=3, fill=black!30] (4) {};
	\node[circle, draw, right of=4, fill=black!30] (5) {};
	\node[circle, draw, right of=5, fill=black!30] (6) {};
	\node[circle, draw, below of=6, xshift=0.5cm, fill=black!30] (7) {};
	\node[circle, draw, left of=7, fill=black!30] (8) {};
	\draw[-, dashed] (1) to[out=135, in=-30] (3);
	\draw[-, dashed] (1) to[out=150, in=-45] (3);
	\path[draw] (1) -- (2) -- (3) -- (4);
	\draw[-, very thick] (4) -- (5);
	\path[draw] (5) -- (6) -- (7) -- (8) -- (1);
	\draw[-, dashed] (8) to[out=110, in=-55] (5);
	\draw[-, dashed] (8) to[out=125, in=-70] (5);
	\draw[-, dashed] (8) to[out=50, in=-105] (6);
	\draw[-, dashed] (8) to[out=65, in=-120] (6);
	
	\node[circle, draw,fill=black, left of=1, xshift=-6cm] (10) {};
	\node[circle, draw,fill=black, right of=10] (11) {};
	\node[circle, draw, left of=10, fill=black!30] (12) {};
	\node[circle, draw, above of=12, xshift=-0.5cm, fill=black!30] (13) {};
	\node[circle, draw, right of=13, fill=black!30] (14) {};
	\node[circle, draw, right of=14, xshift=1cm, fill=black!30] (15) {};
	\node[circle, draw, right of=15, fill=black!30] (16) {};
	\node[circle, draw, below of=16, xshift=0.5cm, fill=black!30] (17) {};
	\node[circle, draw, left of=17, fill=black!30] (18) {};
	\draw[-, dashed] (10) to[out=135, in=-30] (13);
	\draw[-, dashed] (10) to[out=150, in=-45] (13);
	\path[draw] (10) -- (12) -- (13) -- (14);
	\path[draw] (15) -- (16) -- (17) -- (18) -- (11);
	\draw[-, dashed] (18) to[out=110, in=-55] (15);
	\draw[-, dashed] (18) to[out=125, in=-70] (15);
	\draw[-, dashed] (18) to[out=50, in=-105] (16);
	\draw[-, dashed] (18) to[out=65, in=-120] (16);
	\draw[-, very thick] (10) to (14);
	\draw[-, very thick] (11) to (15);
	
	\coordinate[left of=1, xshift=-2cm, yshift=0.5cm] (100);
	\coordinate[right of=100, xshift=-0.5cm] (101);
	\draw[->, very thick] (100) to (101);

	\end{tikzpicture}
	\caption{Vertex-edge-identification of two Eulerian graphs}%
	\label{fig:vertex-edge-identification}%
\end{figure}

If $e_i$ and $u_i$ are clear from the context or the statement is independent from the choice of $e_i$ and~$u_i$ then we simply write $G_1 \vident G_2$, $G_1 \eident G_2$ and $G_1 \veident G_2$.

\paragraph{Cycles invariants are compatible with the identification operations}
In the following we demonstrate that the identification operations preserve the cycle behaviour in a natural way.
We just keep all cycles whose edges are untouched by the operation (in the case of vertex identification these are all cycles).
In each of $G_1$ and $G_2$ exactly one edge is deleted in the construction of $G_1 \eident G_2$ respectively $G_1 \veident G_2$.
We obtain a cycle in $G_1 \eident G_2$ respectively $G_1 \veident G_2$ which uses the edges not contained in $E(G_1)$ nor in~$E(G_2)$ by combining a cycle from $G_1$ with a cycle from $G_2$ each containing a deleted edge.

\begin{lemma}[Cycle invariants under construction operations]
	\label{lem: operations preserve cycles}
	Let $G_1$ and $G_2$ be Eulerian graphs.
	\begin{enumerate}[(i)]
		\item The invariants $c$ and $\nu$ show the following behaviour under vertex-identification:
		\begin{align*}
		c(G_1\vident G_2) &= c(G_1) + c(G_2),\\
		\nu(G_1\vident G_2) &= \nu(G_1) + \nu(G_2).
		\end{align*}
		\item They behave in the following way under edge-identification and vertex-edge-identification:
		\begin{align*}
		c(G_1\eident G_2) = c(G_1\veident G_2) &= c(G_1) + c(G_2) -1 ,\\
		\nu(G_1\eident G_2) = \nu(G_1\veident G_2) &= \nu(G_1) + \nu(G_2) -1.
		\end{align*}
	\end{enumerate}
\end{lemma}

\begin{proof}\hfill
	\begin{enumerate}[(i)]
		\item 	For $i \in \{1,2\}$ let $v_i \in V(G_i)$ such that $G_1 \vident G_2 = (G_1, v_1) \vident (G_2, v_2)$.
		The vertex which arises from the identification of $v_1\in V(G_1)$ and $v_2\in V(G_2)$ is a cut-vertex.
		Thus, we obtain a one-to-one-correspondence of cycle decompositions in~$G_1 \cup G_2$ and cycle decompositions in $G_1 \vident G_2$ just by relabelling $v_1$ and $v_2$ to~$v$ and adjusting the incident edges.
		Altogether we obtain 
		\[c(G_1 \vident G_2) = c(G_1) + c(G_2)~\text{and}~\nu(G_1 \vident G_2) = \nu(G_1) + \nu(G_2).\]
		\item Let $G_1 \eident G_2 = (G_1, e_1, u_1) \eident (G_2, e_2, u_2)$ for suitable $e_i \in E(G_i)$ and $u_i \in V(G_i)$, $i \in \{1,2\}$.
		In a cycle decomposition of $G_i$ there is exactly one cycle~$C_i$ containing the edge  $e_i$ with end vertices $u_i$ and $v_i$.
		We obtain a one-to-one-correspondence of cycle decompositions in $G_1 \cup G_2$ and cycle decompositions in~$G_1 \eident G_2$ by keeping all cycles from the decompositions of $G_1$ and $G_2$ except $C_1$ and $C_2$ and adding the cycle $C$ with $E(C) = (E(C_1) \cup E(C_2)\setminus \{u_1v_1, u_2v_2\}) \cup \{u_1u_2, v_1v_2\}$, see Figure \ref{fig:Edge-identification}.
		Thus,
		\begin{align*}
		c(G_1 \eident G_2) &= c(G_1 \cup G_2) - 1 = c(G_1) + c(G_2) - 1~\text{and}\\
		\nu(G_1 \eident G_2) &= \nu(G_1 \cup G_2) - 1 = \nu(G_1) + \nu(G_2) - 1.
		\end{align*}

		Now let $G_1 \veident G_2 = (G_1, e_1, u_1) \veident (G_2, e_2, u_2)$ for suitable $e_i \in E(G_i)$, $u_i \in V(G_i)$, $i \in \{1,2\}$.
		Analogously to the previous operation, we obtain a one-to-one correspondence between cycle decompositions of $G_1 \cup G_2$ and $G_1 \veident G_2$ by choosing $C$ with $E(C) = (E(C_1) \cup E(C_2) \cup \{u_1u_2\}) \setminus \{e_1,e_2\}$, see Figure~\ref{fig:vertex-edge-identification}.
		Consequently we obtain the same relations as above:
		\[{c(G_1 \veident G_2)} = c(G_1) + c(G_2) - 1~\text{and}~\nu(G_1 \veident G_2) = \nu(G_1) + \nu(G_2) - 1.\]	\qedhere	
	\end{enumerate}
\end{proof}

\begin{corollary}
	Let $G_1$, $G_2$ be Eulerian graphs. 
	If $G=G_1\circ G_2$ for some~${\circ\in\{\vident,\eident,\veident\}}$ then it holds true that
	\begin{align*}
	\nu(G)-c(G)=\left(\nu(G_1)-c(G_1)\right) + \left(\nu(G_2) - c(G_2)\right).
	\end{align*}
\end{corollary}

\medskip

\paragraph{Connectivity is compatible with the identification operations}
We show in Lemma \ref{lem: separating sets} that the behaviour of paths between two given vertices in $G_1$ is preserved in $G_1 \veident G_2$.
We follow the intuition to keep all paths which do not contain the edge of $G_1$ which is deleted in $G_1 \veident G_2$ and to reroute a path which uses the deleted edge along a path in $G_2$. We translate the results to the construction $G_1 \eident G_2$.
We start off by the observation that cut-edges are preserved under vertex-edge-identification.

\begin{observation}
	\label{obs: veident cutedges}
	Let $G_1$ be a graph containing a cut-edge and let $G_2$ be some other graph.
	Then, also $G_1 \veident G_2$ contains a cut-edge.
\end{observation}

\begin{proof}
	Let $e_i \in E(G_i)$ and $u_i \in V(G_i)$  for $i \in \{1,2\}$ such that $G_1 \veident G_2 = (G_1, e_1, u_1) \veident (G_2, e_2, u_2)$.
	Let $e' \in E(G_1)$ be a cut-edge.
	If $e' \neq e_1$ then $e'$ is still a cut-edge in $G_1 \veident G_2$.
	Otherwise, the new edge connecting $u_1$ and $u_2$ is a cut-edge in~$G_1 \veident G_2$.
\end{proof}

\begin{lemma}
	\label{lem: separating sets}
	Let $G_1, G_2$ be 2-edge-connected graphs with edges $e_i=v_iu_i\in E(G_i)$ for $i\in \{1,2\}$.
	Let $e$ be the edge in $(G_1,e_1, u_1)  \veident (G_2, e_2, u_2)$ with end vertices~$u_1$ and $u_2$ 
	and let $v$ be the vertex arising from the identification of $v_1$ and $v_2$.
	Let $S \subseteq V(G_1) \cup E(G_1)$. Further set
	\begin{align*}
	S' \coloneqq \begin{cases}
	\phantom{(}S &\text{if~} v_1, e_1 \notin S, \\
	(S\setminus \{e_1\}) \cup \{e\} &\text{if~} e_1 \in S, v_1 \notin S, \\
	(S\setminus \{v_1\}) \cup \{v\} &\text{if~} v_1 \in S, e_1 \notin S,  \\
	(S\setminus \{e_1, v_1\}) \cup \{e, v\} &\text{if~} v_1, e_1 \in S.
	\end{cases}
	\end{align*}
	Let $w_1, w_2 \in V(G_1)$ be two distinct vertices.
	We may assume $w_1 \neq v_1$.
	Set
	\begin{align*}
	w_2' \coloneqq \begin{cases}
	w_2 &\text{if~} w_2 \neq v_1, \\
	v &\text{if~} w_2 = v_1.
	\end{cases}
	\end{align*} 
	The set $S$ separates $w_1$ and $w_2$ in $G_1$ if and only if $S'$ separates $w_1$ and $w_2'$ in $G_1 \veident G_2$.
\end{lemma}

\begin{proof}
	If suffices to show: $G_1$ contains a $w_1$-$w_2$-path without elements from $S$ if and only if $G_1 \veident G_2$ contains a $w_1$-$w_2'$-path without elements from $S'$.
	
	\medskip
	\noindent
	Let $P$ be a $w_1$-$w_2$-path in $G_1$ with $(V(P) \cup E(P)) \cap S = \emptyset$.
	Assume that $P$ does not contain $v_1$ and~$e_1$, then $w_2' = w_2$ and $P$ is a $w_1$-$w_2'$-path in $G_1\veident G_2$ with $(V(P) \cup E(P)) \cap S' = \emptyset$.
	Now assume that $P$ contains $v_1$ but not $e_1$. 
	We obtain a $w_1$-$w_2'$-path~$P'$ with $(V(P) \cup E(P)) \cap S' = \emptyset$ by renaming~$v_1$ to $v$ in $P$.
	Last assume that~$P$ contains $e_1$. Then, $P$ is of the form $P = P_1u_1e_1v_1P_2$ where $P_1$ is a $w_1$-$u_1$-path (resp.\ $w_2$-$u_1$-path) and $P_2$ is a $v_1$-$w_2$-path (resp.\ $v_1$-$w_1$-path) in $G_1$.
	From the 2-edge-connectivity of $G_2$, we obtain that there exists a $u_2$-$v_2$-path $Q$ in $G_2-e_2$.
	Let $Q'$ be the path obtained from $Q$ by renaming $v_2$ to $v$ and let $P_2'$ be the path in $G_1 \veident G_2$ obtained from $P_2$ by renaming $v_1$ to $v$.
	Now $P_1u_1eu_2Q'P_2'$ is a  $w_1$-$w_2$-path in $G_1 \veident G_2$ with
	$(V(P_1Q'P_2') \cup E(P_1Q'P_2')) \cap S' = \emptyset$.
	
	\medskip
	\noindent
	Let now $P'$ be a $w_1$-$w_2^\prime$-path in $G_1 \veident G_2$ with $(V(P') \cup E(P')) \cap S' = \emptyset$.
	If $V(P^\prime)\subseteq V(G_1)\cup\{v\}$ then the path obtained from $P^\prime$ by renaming $v$ to $v_1$ (if it is contained in~$P^\prime$) is a $w_1$-$w_2$-path in $G_1$.
	Otherwise $P'$ must be of the form $P' = P_1'u_1eu_2P_2'P_3'$, where $P_1'$ is a $w_1$-$u_1$-path (resp.\ $w_2$-$u_1$-path) with edges in $E(G_1)\setminus\{e_1\}$, $P_2'$ is a $u_2$-$v$-path with edges in $E(G_2) \setminus \{e_2\}$ and $P_3'$ is a $v$-$w_2$-path (resp.\ $v$-$w_1$-path) with edges in $E(G_1)\setminus\{e_1\}$.
	Let $P_3$ be the path in $G_1$ that arises from $P_3'$ by renaming $v$ to $v_1$.
	We obtain $(V(P_1'u_1e_1v_1P_3) \cup E(P_1'u_1e_1v_1P_3)) \cap S = \emptyset$ and $P_1^\prime u_1e_1v_1P_3$ is a $w_1$-$w_2$-path in $G_1$.
\end{proof}

\begin{corollary}
	\label{coro: biconnectivity}
	Let $G_1$ and $G_2$ be graphs with edges $e_i=v_iu_i\in E(G_i)$ for $i=1,2$.
	Let $e$ be the edge in $(G_1,e_1, u_1)  \veident (G_2, e_2, u_2)$ with end vertices $u_1$ and $u_2$,
	let $v$ be the vertex arising from the identification of $v_1$ and $v_2$.
	It holds that $G_1 \veident G_2$ is biconnected if and only if $G_1$ and $G_2$ are biconnected and contain more than one edge.
\end{corollary}

\begin{proof}
	Assume that $G_1$ and $G_2$ are both biconnected and each contain more than one edge.
	Let $w \in V(G_1)$ and $x \in V(G_2)$.
	Then, by Menger's Theorem (see \cite{West01}) there exist internally vertex disjoint paths $P_1$ from $w$ to $u_1$ and $Q_1$ from $w$ to $v_1$ in $G_1$. Further, there exist internally vertex disjoint paths $P_2$ from $u_2$ to $x$ and $Q_2$ from $v_2$ to~$x$ in $G_2$.
	Let for $i\in\{1,2\}$ $Q_i'$ be the path that arises from $Q_i$ by renaming $v_i$ to $v$. Now $P_1u_1eu_2P_2$ and $Q_1'Q_2'$ are two internally vertex disjoint $w$-$x$-paths in $G_1 \veident G_2$.
	For $i\in\{1,2\}$ and two vertices $w_1$ and $w_2$  in $V(G_i)\setminus\{v_i\}$ we obtain from Lemma \ref{lem: separating sets} and Menger's Theorem that there exists two internally vertex disjoint paths in~$G_1 \veident G_2$ connecting $w_1$ and $w_2$.
	
	If $G_i$ for some $i \in \{1,2\}$ contains just one edge, then $G_1 \veident G_2$ contains a cut vertex.
	Next suppose that $G_i$ has a cut-edge for some $i \in \{1,2\}$.
	By Observation \ref{obs: veident cutedges} also $G_1 \veident G_2$ has a cut-edge.
	Last suppose that $G_1$ and $G_2$ are 2-edge-connected and~$G_i$ has a cut-vertex for some~$i\in\{1,2\}$.
	But then by Lemma \ref{lem: separating sets} also $G_1 \veident G_2$ has a cut-vertex.
	This settles the claim.
\end{proof}

\begin{lemma}
	\label{lem: connectivity and edge identification}
	Let $G_1$ and $G_2$ be 2-edge connected graphs and let $e_i \in E(G_i)$ with end vertices $u_i$ and $v_i$ for $i \in \{1,2\}$.
	Let $S \subseteq V(G_1) \cup E(G_1)$.
	Let $w_1, w_2 \in V(G_1)$ be two distinct vertices.
	Set
	\begin{align*}
	S' \coloneqq \begin{cases}
	S~&\text{if}~e_1 \notin S,\\
	(S \setminus \{e_1\}) \cup \{u_1u_2\}~&\text{if}~e_1 \in S.
	\end{cases}
	\end{align*}
	Then, $S$ separates the vertices $w_1$ and $w_2$ in $G_1$ if and only if $S'$ separates $w_1$ and $w_2$ in $(G_1, e_1, u_1) \eident (G_2, e_2, u_2)$.
	In particular, $G_1 \eident G_2$ is biconnected if and only if $G_1$ and $G_2$ are biconnected.
\end{lemma}

\begin{proof}
	The proof is analogous to the proofs of Lemma \ref{lem: separating sets} and Corollary~\ref{coro: biconnectivity}.
\end{proof}

\medskip
\paragraph{Treewidth is compatible with the identification operations}
Also the treewidth behaves nicely with the identification operations.
Clearly, the treewidth of a graph can be computed knowing the treewidth of its biconnected components.
Furthermore, a width-optimal tree decomposition of~${G_1 \eident G_2}$ or $G_1 \veident G_2$ can be constructed by just slighlty changing a tree decomposition of $G_1 \cup G_2$.
The results are summarized in Lemma \ref{lemma: treewidth and identification operations}.

\begin{lemma}
	\label{lemma: treewidth and identification operations}
	Let $G_1$ and $G_2$ be $2$-edge-connected graphs. It holds true that
	\begin{align*}
	\tw(G_1\vident G_2) &= \max\{\tw(G_1), \tw(G_2)\},\\
	\tw(G_1\eident G_2) &= \max\{2, \tw(G_1), \tw(G_2)\} ~\text{and} \\
	\tw(G_1\veident G_2) &= \max\{2, \tw(G_1), \tw(G_2)\}.
	\end{align*}
\end{lemma}

\begin{proof}
	A graph is of treewidth at most $k$ if and only if all of its biconnected components are of treewidth at most $k$, cf.\ \cite{Bodlaender98}.
	Thus, $\tw(G_1 \vident G_2) = \max\{\tw(G_1), \tw(G_2) \}$.
	
	By the assumption that $G_1$ and $G_2$ are $2$-edge connected we obtain that $G_1 \eident G_2$ and $G_1 \veident G_2$ each contain a cycle of length not less than 3. Consequently $\tw(G_1 \eident G_2) \geq 2$ and $\tw(G_1 \veident G_2) \geq 2$.
	Now, $G_1$ and $G_2$ are minors of $G_1 \eident G_2$ and $G_1 \veident G_2$. Altogether $\tw(G_1\eident G_2) \geq \max\{2, \tw(G_1), \tw(G_2)\}$ and $\tw(G_1\veident G_2) \geq \max\{2, \tw(G_1), \tw(G_2)\}$.
	For the other inequality let $(\mathcal{T}^{(i)}, \mathcal{B}^{(i)})$ be a tree decomposition of $G_i$ and let $B_i \in \mathcal{B}^{(i)}$ be a bag with $\{u_i, v_i\} \in B_i$ for $i \in \{1,2\}$.
	We obtain a tree decomposition of $G_1 \eident G_2$ of width $\max\{2, \tw(G_1), \tw(G_2)\}$ by the following construction. Set
	\begin{align*}
	B_a &\coloneqq \{u_1, u_2, v_1\},\\
	B_b &\coloneqq \{u_2, v_1, v_2\},\\
	\mathcal{B} &\coloneqq \mathcal{B}^{(1)} \cup \mathcal{B}^{(2)} \cup \{B_a, B_b\},\\
	V(\mathcal{T}) &\coloneqq V(\mathcal{T}^{(1)}) \cup V(\mathcal{T}^{(2)}) \cup \{a, b\}~\text{and}\\
	E(\mathcal{T})&\coloneqq E(\mathcal{T}^{(1)}) \cup E(\mathcal{T}^{(2)})) \cup \{1a, ab, b2 \}.
	\end{align*}
	Now $(\mathcal{T}, \mathcal{B})$ is a tree decomposition of $G_1 \eident G_2$ of width $\max\{2, \tw(G_1), \tw(G_2)\}$.
	The inequality for $G_1 \veident G_2$ follows immediately since $G_1 \veident G_2$ is a minor of $G_1 \eident G_2$.
	This settles the claim.
\end{proof}

\section{Subquartic Eulerian graphs of treewidth at most 2}
\label{sec: subquartic}
We are now ready to discuss a constructive characterization starting with a simple class of base graphs -- the closed necklaces -- and then only using the operators $\eident$ and $\vident$. 
More precisely we characterize all Eulerian graphs
with treewidth at most $2$ and maximum degree 4. 
A \emph{closed necklace} is a graph which can be constructed from a cycle of length at least $2$ by duplicating all of its edges.
We define the class $\mathcal{H}$ recursively:

\begin{itemize}
	\item All closed necklaces are contained in $\mathcal{H}$.
	\item If $H_1, H_2 \in \mathcal{H}$, then also $H_1 \eident H_2 \in \mathcal{H}$.
\end{itemize}

\begin{observation}
	\label{obs: necklace}
	The only biconnected 4-regular graph of treewidth 1 is the closed necklace on two vertices.
	The only biconnected 4-regular graph on three vertices is the closed necklace on three vertices.
\end{observation}

\begin{lemma}
	\label{lem: 2cut in non necklaces}
	Let $G$ be a biconnected 4-regular graph of treewidth 2 which is not isomorphic to a closed necklace.
	Then $G$ has a 2-cut $\{e_1, e_2\}$ where no end vertex of $e_1$ coincides with an end vertex of $e_2$.
\end{lemma}

\begin{proof}
	We prove the following statement by induction on the number of vertices of~$G$:
	A biconnected 4-regular graph of treewidth 2 is either a closed necklace or it has a 2-cut $\{e_1, e_2\}$ where no end vertex of $e_1$ coincides with an end vertex of $e_2$.
	The base case $|V(G)| \leq 3$ is settled by Observation \ref{obs: necklace}.
	Let now $|V(G)| \geq 4$.
	
	\medskip
	\noindent
	Suppose that $G$ contains a vertex $u$ with $N_G(u) = \{x_1, x_2\}$ such that $u$ is connected to each $x_i$ with exactly two edges.
	We construct a graph $G'$ by removing $u$ and adding two edges between $x_1$ and $x_2$.
	Observe that $G'$ is still biconnected, 4-regular, of treewidth at most 2.
	By induction $G'$ is either a closed necklace -- in this case $G$ is also a closed necklace.
	Or $G'$  contains a two-edge-cut of the desired form, then it is also a cut of the desired form in $G$.
	
	\medskip
	\noindent
	Now suppose that each vertex in $G$ which has exactly two neighbours is connected to one of them with three edges and to the other one with a single edge.
	Let $(\{X_i \colon i \in I\}, T)$ be a smooth tree decomposition of $G$ of width 2.
	Let $l$ be a leaf in $T$ with unique neighbour $k$, which exists as~$\tw(G)=2$ and $V(G)\geq 4$.
	As the tree decomposition is smooth we have $X_l=\{u, x_1, x_2\}$ and~$X_k=\{v, x_1, x_2\}$ with distinct vertices $u, v, x_1, x_2\in V(G)$.
	The biconnectivity of $G$ and the structure of the bags $X_l$ and~$X_k$ imply $N_G(u)=\{x_1, x_2\}$.
	We may assume that there are three edges connecting $u$ to one of its neighbours, say $x_1$.
	Let $N_G(x_1)=\{u, x_1^\prime\}$ for some~${x_1^\prime\in V(G)}$. 
	Note that $x_1^\prime\neq x_2$ as otherwise $x_2$ would be a cut-vertex, contradicting the fact that $G$ is biconnected.
	Thus, $\{x_1,x_1^\prime, ux_2\}$ is a $2$-cut of the desired form in $G$. 
\end{proof}

\begin{theorem}
	\label{thm: 4reg tw2}
	Let $G$ be a graph. Then $G\in\mathcal{H}$ if and only if it is a biconnected $4$-regular graph of 
	treewidth at most $2$.
\end{theorem}

\begin{proof}
	Let $G \in \mathcal{H}$. 
	Note that this implies that $G$ is $2$-edge-connected by Lemma \ref{lem: connectivity and edge identification}. 
	If $G$ is a closed necklace, it is biconnected, 4-regular and fulfils $\tw(G) \leq 2$.
	We prove that $G$ fulfils the desired properties by induction on the number of vertices.
	If~$V(G)=2$ it is a closed necklace.
	So assume that $V(G)\geq 3$ and $G$ is not a closed necklace.
	Consequently $G = G_1 \eident G_2$ for two graphs~$G_1, G_2\in\mathcal{H}$.
	By induction~$G_1$ and $G_2$ are biconnected, $4$-regular and have treewidth at most $2$.
	Then also $G$ is 4-regular, $\tw(G) \leq 2$ by Lemma \ref{lemma: treewidth and identification operations} and $G$ is biconnected by Lemma \ref{lem: connectivity and edge identification}.
	
	\medskip \noindent
	Let now $G$ be biconnected $4$-regular of treewidth at most $2$.
	We prove $G\in\mathcal{H}$ by induction on~$|V(G)|$.
	If $|V(G)| \in \{2,3\}$ then $G$ must be the necklace with two or three vertices by Observation~\ref{obs: necklace} and thus $G \in \mathcal{H}$.
	Let now $|V(G)| \geq 4$.
	If $G$ is not a closed necklace, then $\tw(G) = 2$ by Observation \ref{obs: necklace}. We may apply Lemma \ref{lem: 2cut in non necklaces} and obtain that $G = G_1 \eident G_2$ for suitable $G_1, G_2$.
	Observe that $G_1$ and $G_2$ are biconnected, cf.~Lemma~\ref{lem: connectivity and edge identification}, and 4-regular. Furthermore, their treewidth is bounded by 2 since they are minors of $G$.
	By induction, $G_1, G_2 \in \mathcal{H}$. Thus, also $G \in \mathcal{H}$.
\end{proof}

Let $v$ be a cut-vertex in a 4-regular graph $G$.
The degree of $v$ in the biconnected components of $G$ is 2 since otherwise the edges incident to $v$ would contain an odd cut in an Eulerian graph.
Consequently, all degrees of vertices in biconnected components of $G$ lie in $\{2,4\}$.
We conclude that a biconnected component of a 4-regular graph is either a cycle or can be obtained from a biconnected 4-regular graph by subdivision.
Together with Theorem \ref{thm: 4reg tw2} we obtain:

\begin{corollary}
	\label{cor:4reg}
	A connected graph $G$ is $4$-regular of treewidth at most $2$ if and only if each of its biconnected components $H$ is either a cycle such that each of its vertices is a cut-vertex in $G$ or the graph obtained by successively resolve all former cut-vertices in $H$ is contained in $\mathcal{H}$.
\end{corollary}

We obtain a constructive characterization of the class $\mathcal{H}^\prime$ containing all Eulerian graphs of treewidth at most 2 and with maximum degree 4 in a straight forward way:
\begin{itemize}
	\item All closed necklaces are in $\mathcal{H}'$.
	\item All cycles are in $\mathcal{H}'$.
	\item If $G \in \mathcal{H}'$ and $G'$ is obtained from $G$ by subdividing an edge then $G' \in \mathcal{H}'$.
	\item If $G_1, G_2 \in \mathcal{H}'$, then $G_1 \eident G_2 \in \mathcal{H}'$.
	\item If for $i \in \{1,2\}$ $G_i \in \mathcal{H}'$ and $v_i \in V(G_i)$ with $\deg_{G_i}(v_i) = 2$, then $(G_1, v_1)\vident (G_2, v_2) \in \mathcal{H}'$.
\end{itemize}

Now that we have extensively studied the applications of the binary operator $\eident$, we continue with considering the 
class of graphs which arises using the operators $\vident$ and $\veident$.

\section{Graphs with unique cycle-decomposition size}
\label{sec: nu = c}
In this section we prove our main result -- two equivalent characterizations for the class of graphs where the minimum and maximum number of cycles in a cycle decomposition coincide.
We show first that the class of graphs with unique cycle decomposition size is contained in the class of graphs where two cycles intersect at most twice.

\begin{lemma}
	\label{lem:2cyc} \hfill
	\begin{enumerate}[(i)]
		\item Let $H$ be an Eulerian subgraph of an Eulerian graph $G$. If $c(H) < \nu(H)$ then $c(G) < \nu(G)$.
		\item Let $H'$ be a graph which is decomposable into two edge disjoint cycles that have more than two vertices in common. 
		Then $c(H') = 2$ and $\nu(H')~\geq~3$.
	\end{enumerate}
	In particular: An Eulerian graph $G$ containing two edge-disjoint cycles that have more than two vertices in common satisfies $c(G) < \nu(G)$.
\end{lemma}

\begin{proof}
	Let $\overline{\mathcal{C}}$ be a maximum cycle decomposition of $H$ and $\underline{\mathcal{C}}$ be a minimum cycle decomposition of $H$.
	Further let $\mathcal{C}$ be a cycle decomposition of $G - E(H)$. We obtain
	\begin{align*}
	c(G) \leq |\mathcal{C} \cup \mathcal{\underline{C}}| < |\mathcal{C} \cup \overline{\mathcal{C}}| \leq \nu(G),
	\end{align*}
	proving the first claim.
	
	Let now $H'=C_1\cup C_2$ for two edge-disjoint cycles $C_1, C_2$. 
	Further let $v_1, v_2, v_3\in V(C_1)\cap V(C_2)$ be three distinct vertices. 
	Let $i\in\{1,2\}$. 
	As $C_i$ is a cycle there exists a path $P_i$ from~$v_1$ to $v_2$ with~$v_3\notin V(P_i)$, which is a subgraph of $C_i$.
	Then $P_1P_2$ is even and the degree of $v_3$ in~${H'-E(P_1P_2)}$ is $4$. 
	We get $\nu(H')\geq \nu(H'-E(P_1P_2)) + \nu(P_1P_2) \geq 2 + 1 = 3$ as claimed.
\end{proof}

We are now ready to present a constructive characterization of all Eulerian graphs with the property that the number of cycles in a cycle decomposition is unique.
Let us define a class of graphs $\mathcal{G}$, where the base graphs are Eulerian multiedges and all other graphs recursively arise from operations on two disjoint graphs in the class.
\begin{itemize}
	\item If $G$ is an Eulerian multiedge, i.e.\ a graph that consist only of two vertices and an even number of parallel edges between the two vertices, then $G \in \mathcal{G}$.
	\item Let $G_1, G_2 \in \mathcal{G}$ with $V(G_1) \cap V(G_2) = \emptyset$ and $v_i$ a vertex in $G_i$ for $i \in \{1,2\}$. Then $(G_1,v_1) \vident (G_2,v_2) \in \mathcal{G}$.
	\item Let $G_1, G_2 \in \mathcal{G}$ with $V(G_1) \cap V(G_2) = \emptyset$, $e_i$ be an edge in $G_i$ from $u_i$ to $v_i$ for $i \in \{1,2\}$. Then $(G_1, e_1, u_1) \veident (G_2, e_2, u_2) \in \mathcal{G}$.
\end{itemize}

\medskip

\begin{theorem} \label{thm: main}
	Let $G$ be a graph. The following three statements are equivalent.
	\begin{enumerate}[(i)]
		\item \label{itm: cG = nuG} $G$ is Eulerian with $c(G) = \nu(G)$.
		\item \label{itm: G2cactus} $G$ is Eulerian and no two edge disjoint cycles in $G$ have more than two vertices in common.
		\item \label{itm: G in mathcalG} $G \in \mathcal{G}$.
	\end{enumerate}
\end{theorem}

\begin{proof}
	\noindent \eqref{itm: cG = nuG} implies \eqref{itm: G2cactus}: 
	This implication is stated in Lemma \ref{lem:2cyc}.

	\bigskip
	
	\noindent \eqref{itm: G2cactus} implies \eqref{itm: G in mathcalG}:
	Suppose that there are graphs satisfying \eqref{itm: G2cactus} but not \eqref{itm: G in mathcalG}.
	Then amongst those graphs there exists a graph $G$ of lowest order.
	Note that $G$ is not an Eulerian multiedge, since these satisfy \eqref{itm: G in mathcalG}.
	We establish further structural properties of $G$:
	\begin{property}\label{prop 0}
		$G$ is biconnected.
	\end{property}
	\noindent \emph{Proof of Property \ref{prop 0}:}
	Suppose that $G$ is not biconnected. Then there exists some cut-vertex $v\in V(G)$. Thus,
	there are two graphs $G_1$ and $G_2$ such that $G=G_1\vident G_2$. As no two cycles in 
	$G_1$ and $G_2$ have more than two vertices in common, we get $G_1,G_2\in\mathcal{G}$ by the minimality of $G$ and
	thereby $G\in\mathcal{G}$, contradicting the choice of $G$.
	\begin{property}\label{prop I}
		For all $e \in E(G)\colon G-e$ is biconnected.
	\end{property}
	\noindent \emph{Proof of Property \ref{prop I}:}
	Suppose that $G-e$ is not biconnected. Then there are two graphs~$G_1$ and $G_2$ such that $G = G_1 \veident G_2$.
	Note that the vertex $v\in V(G)$ that is split up in $G_1$ and $G_2$ cannot be an endpoint of $e$, as $G$
	is biconnected by Property \ref{prop 0}.
	Further observe that neither $G_1$ nor $G_2$ contain edge disjoint cycles with more than two vertices in common.
	By the minimality of $G$ we obtain $G_1, G_2 \in \mathcal{G}$.
	We obtain $G \in \mathcal{G}$. A contradiction.
	
	\begin{property}\label{prop II}
		For all $v \in V(G)\colon G-v$ is 2-edge-connected.	
	\end{property}
	\noindent \emph{Proof of Property \ref{prop II}:}
	Suppose that $G-v$ contains a one-edge-separator.
	Again, there are two graphs $G_1$ and $G_2$ such that $G = G_1 \veident G_2$ and we can argue as in the proof of Property \ref{prop I}.
	
	\begin{property}\label{prop III}
		For all $v \in V(G)$ there is at most one neighbour of $v$ that is connected to $v$ by multiple edges.
	\end{property}
	\noindent \emph{Proof of Property \ref{prop III}:}
	Assume that there is a vertex $v$ that is connected to two different vertices $w_1$ and $w_2$ by multiple edges. By Property \ref{prop II} we know that $G-v$ is $2$-edge-connected.
	By Menger's Theorem (see \cite{West01}) there exist edge disjoint paths $P_1$, $P_2$ from $w_1$ to $w_2$ in $G-v$.
	But then the two cycles $vw_1P_1w_2v$ and $vw_1P_2w_2v$ are edge disjoint and share more than two vertices.
	This is a contradiction to (\ref{itm: G2cactus}).
	
	\begin{property}\label{prop IV}
		For all $v \in V(G)$ we have $|N(v)| \geq 4$.
	\end{property}
	\noindent \emph{Proof of Property \ref{prop IV}:}
	Suppose there is a vertex $v$ with $|N(v)|\leq 3$.
	First we assume that $|N(v)|=1$. Then $G$ is either an Eulerian multiedge or not biconnected which is a contradiction to the assumption respectively Property \ref{prop 0}.
	Now assume that $|N(v)|=2$, say $N(v) = \{w_1,w_2\}$.
	By Property \ref{prop III} $v$ cannot be connected to both neighbours by multiple edges, say $v$ is connected to $w_1$ by a single edge $e$. If we delete $w_2$ from $G-e$ we isolate $v$ which is a contradiction to Property \ref{prop I}.
	
	Last assume that $|N(v)| = 3$, say $N(v) = \{w_1, w_2, w_3\}$.
	By Property \ref{prop III}, we may further assume that $v$ is connected to $w_1$ and $w_2$ by a single edge only.
	By Property \ref{prop I} the graph $G-vw_1$ is biconnected.
	Thus, there is a cycle $C$ in $G-vw_1$ containing the vertices $v$ and $w_1$.
	Since $w_2$ and~$w_3$ are the only neighbours of $v$ in $G-vw_1$, we obtain that $C$ also contains the vertices $w_2$ and~$w_3$.
	The graph~$G-E(C)$ is even and consequently $vw_1$ is contained in some cycle $C'$ in $G-E(C)$.
	The single edge~$vw_2$ is contained in $C$.
	Hence, $v$ has only neighbours $w_1$ and $w_3$ in $G-E(C)$.
	Thus, $C'$ contains the vertex~$w_3$ as well.
	Thereby $C$ and $C'$ are two edge disjoint cycles with more than two vertices in common -- a contradiction.
	
	\medskip
	
	\noindent We exploit Properties \ref{prop III} and \ref{prop IV} to complete the proof: Regard a path $P=v_1 v_2\dots v_k$ with the property that $N(v_k)\subseteq\{v_1,\dots,v_{k-1}\}$ and $v_1v_k\in E$.
	Such a path can be found in a greedy fashion: 
	Start at some vertex $v$ in the graph and always move to a new vertex until all neighbours of the current vertex $w$ have already been visited.
	The resulting path contains the neighbourhood of $w$. 
	Now simply set $v_1$ to be the neighbour of $w$ that has been visited first and the subsequent vertices accordingly.
	By Property \ref{prop IV} we have $|N(v_k)|\geq 4$.
	Thus, we can find $i,j\in\{2,..,k-2\}$ with $i\neq j$ and $v_i,v_j\in N(v_k)$.
	Property \ref{prop III} implies that $v_k$ is connected to $v_i$ or $v_j$ by a single edge.
	Without loss of generality let this be $v_i$.
	Set $C\coloneqq v_1v_2...v_kv_1$.
	Then $G-E(C)$ is an even graph and we can find a cycle $C'$ in $G-E(C)$ containing the edge $v_kv_i$.
	Since $N(v_k)\subseteq\{v_1,\dots,v_{k-1}\}$ the two edge disjoint cycles $C$ and $C'$ have more than two vertices in common, which contradicts assumption \eqref{itm: G2cactus}.
	Altogether $G\in\mathcal{G}$.
	
	\bigskip
	
	\noindent \eqref{itm: G in mathcalG} implies \eqref{itm: cG = nuG}:
	Eulerian multiedges fulfil property \eqref{itm: cG = nuG}.
	For~$i \in \{1,2\}$ let $G_i$ be a graph that satisfies $c(G_i) = \nu(G_i)$ 
	If $G$ arises from vertex-identification or vertex-edge-identification from graphs $G_1$ and $G_2$, by Lemma \ref{lem: operations preserve cycles} we have
	$$\nu(G)-c(G)=\nu(G_1)-c(G_1) + \nu(G_2) - c(G_2) = 0,$$ 
	which implies \eqref{itm: cG = nuG}.
\end{proof}

Combining the constructive characterization in Theorem \ref{thm: main} with Lemma \ref{lemma: treewidth and identification operations} we obtain that all graphs with unique cycle number are of treewidth at most 2.
In particular, they are planar and at most 2-vertex-connected.

\section{Recognition of graphs with unique cycle number}
\label{sec: algo}
In this section, we present an $\mathcal{O}(n(m+n))$-algorithm which decides if the cycle number of a given Eulerian graph is unique.
The main idea of the algorithm is to exploit the following two observations:

\begin{observation}
	Cycles are subgraphs of the biconnected components of a given graph.
	Hence:
	A graph $G$ fulfils $c(G) = \nu(G)$ if and only if this equation holds true for each of its biconnected components.
\end{observation}

\begin{observation}
	Let $G$ be a biconnected graph. Then $G$ fulfils $c(G) = \nu(G)$ if and only if it fulfils one of the following two properties:
	\begin{itemize}
		\item The graph $G$ is an Eulerian multiedge.
		\item There exists graphs $G_1$ and $G_2$ such that $G = G_1 \veident G_2$. For any two graphs $G_1$ and $G_2$ satisfying this equation it holds that $c(G_1) = \nu(G_1)$ and $c(G_2) = \nu(G_2)$.
	\end{itemize}
\end{observation}

\begin{proof}
	By Theorem \ref{thm: main} $G \in \mathcal{G}$.
	Hence $G$ is either an Eulerian multiedge or there exists $G_1, G_2 \in \mathcal{G}$ with $G = G_1 \vident G_2$ or $G = G_1 \veident G_2$.
	The case $G = G_1 \vident G_2$ cannot occur since $G$ is biconnected.
	Now assume that $G_1, G_2$ are graphs with $G = G_1 \veident G_2$.
	Suppose that $c(G_i) < \nu(G_i)$ for some~$i\in\{1,2\}$.
	We obtain by Lemma \ref{lem: operations preserve cycles} that $c(G) = c(G_1) + c(G_2)-1 < \nu(G_1) + \nu(G_2) -1 = \nu(G)$. A contradiction.
\end{proof}

These two observations already give an outline of the whole algorithm: We start by computing the biconnected components of the given graph.
If a biconnected component is of the form $G_1 \veident G_2$, we replace it by $G_1 \cup G_2$ and check if further decomposition is possible.
Corollary \ref{coro: biconnectivity} ensures us that $G_1$ and $G_2$ are still biconnected - hence, it suffices to replace $G_1 \veident G_2$ by $G_1$ and $G_2$ in the list of biconnected components.
If at some point of the algorithm no component allows for further decomposition, the input graph has a unique cycle number if and only if each of the computed components is an Eulerian multiedge.

\begin{definition}[Vertex-Edge Separation]
	Let $G$ be a disjoint union of biconnected graphs.
	Further let $v \in V(G)$ and $e \in E(G)$ be a vertex and an edge in the same component $H$ of $G$.
	We call the tuple $(v, e)$ a \emph{vertex-edge-separator} in $G$ if $H-v-e$ has more than one component.
	Observe that~$(v,e)$ is a vertex-edge-separator if and only if there exist biconnected graphs $H_1,H_2$ with edges $e_i=u_iv_i \in E(H_i)$ for $i=1,2$, such that $H = (H_1, e_1, u_1) \veident (H_2, e_2, u_2)$ where $v$ is the vertex that arises from the identification of $v_1$ and $v_2$ and $e$ is the edge from $u_1$ to $u_2$ in $H$.
	We call the process of replacing $H$ by $H_1 \cup H_2$ in $G$ a \emph{vertex-edge-separation step}. The constructed graph is called \emph{vertex-edge-separation} of $G$. Observe that the constructed graph is again a disjoint union of biconnected graphs by Corollary \ref{coro: biconnectivity}.
\end{definition}

Before we describe the algorithm we will prove a Lemma showing 
that edges which are not contained in a vertex-edge separator at some step of the 
algorithm will never be contained in a vertex-edge separator.
This proof implies that it suffices to check for each vertex only once whether it is contained in a vertex-edge-separator during the algorithm.
\begin{lemma}
	\label{lem:veseparator}
	Let $G$ be a biconnected graph satisfying $G = G_1 \veident G_2$ for graphs $G_1, G_2$.
	Further let $v\in V(G)$ be some vertex in $G$ which is not contained in any vertex-edge separator of $G$.
	Then~$v$ is not contained in any vertex-edge separator of $G_1$ or $G_2$.
\end{lemma}
\begin{proof}
	The graphs $G_1$ and $G_2$ both are biconnected and consequently also 2-edge-connected by Corollary \ref{coro: biconnectivity}.
	As $v$ is not contained in a vertex-edge separator in $G$ it is either contained in~$G_1$ or~$G_2$. 
	Hence, by Lemma \ref{lem: separating sets} $v$ cannot be contained in a vertex edge separator in $G_1$ or~$G_2$. 
\end{proof}

We are now ready to present a formal algorithm and prove its 
correctness. In the description of Algorithm~\ref{alg:test and decompose} we use
the two \emph{black box procedures} \textsc{FindCutEdge} and \textsc{Split}:

\noindent 
\begin{changemargin}{0.6cm}{0.6cm}
	$\textbf{\textsc{FindCutEdge}(G)}$ returns a cut-edge of $G$ if one exists and \texttt{Nil} else.
	
	\medskip
	\noindent
	$\textbf{\textsc{Split}(G, v)}$ 
	gets a graph $G$ and a cut-vertex $v \in V(G)$ as input.
	Let $G_1$, $G_2$ be graphs satisfying $G=(G_1,v_1)\vident (G_2,v_2)$ where $v$ is
	the vertex that arises from identifying $v_1$ with~$v_2$.
	The procedure returns $G_1 \cup G_2$, $v_1$ and $v_2$.
\end{changemargin}

When implementing the algorithm the two procedures would rather be done at the same time using a slightly 
modified version of the \emph{lowpoint algorithm} for finding
biconnected components by Hopcroft and Tarjan, cf.~\cite{Hopcroft1973}.
We merely state it in the presented way to better catch the intuition behind the algorithm.

\begin{algorithm}
	\begin{algorithmic}[1]{TestAndDecompose$(G, v)$}
		\INPUT{Graph~$G$ and vertex $v\in V(G)$.}
		\OUTPUT{Graph~$G$, vertices $v_1$, $v_2$ }
		\STATE $e =u_1u_2 = \textsc{FindCutEdge}(G-v)$
		\IF{$e$ is not \texttt{Nil}}
		\STATE $G=G-e$
		\STATE $G, v_1, v_2 = \textsc{Split}(G, v)$
		\STATE Add an edge between $u_1$ and $v_1$ to $G$.
		\STATE Add an edge between $u_2$ and $v_2$ to $G$.
		\STATE \RETURN $G$, $v_1$, $v_2$
		\ELSE
		\STATE \RETURN $G$, \texttt{Nil}, \texttt{Nil}
		\ENDIF
	\end{algorithmic}
	\caption{Test if vertex is contained in a vertex-edge separator, if so apply vertex-edge separation}\label{alg:test and decompose} 
\end{algorithm}

\begin{algorithm}
	\begin{algorithmic}[1]{VE-Components$(G)$}
		\INPUT{Biconnected graph~$G$.}
		\OUTPUT{Disjoint union $G$ of biconnected graphs.}
		\STATE $S\coloneqq V$
		\WHILE{$S\neq\emptyset$}
		\STATE Take out arbitrary $v\in S$
		\STATE $G,v_1, v_2 =\textsc{TestAndDecompose}(G, v)$\label{algstep:update}
		\IF{$v_1\neq\texttt{Nil}$}
		\STATE Add $v_1$ and $v_2$ to $S$.
		\ENDIF
		\ENDWHILE
		\STATE\RETURN $G$
	\end{algorithmic}
	\caption{Computation of vertex-edge-components using vertex-edge separation}\label{alg:components} 
\end{algorithm}

\begin{theorem}
	Given a biconnected graph $G$ with $n$ vertices and $m$ edges Algorithm~\ref{alg:components} returns a graph $G^\prime$ that does 
	not contain a vertex-edge separator. The graph $G$ can be obtained from~$G^\prime$ by
	repeated vertex-edge-identification of connected components of $G^\prime$.
	Algorithm~\ref{alg:components} can be implemented
	to run in time $\mathcal{O}(n\cdot(m+n))$.
\end{theorem}
\begin{proof}
	Note that each time a vertex-edge separation step is applied the number of vertices, edges and components 
	of $G$ all increase by $1$. The size of the largest component never increases though and at least 
	one component becomes smaller. Thus, Algorithm~\ref{alg:components} terminates. 
	
	Let $G$ be a biconnected graph and $G^\prime$ the graph
	returned by the algorithm starting with $G$.
	By Lemma~\ref{lem:veseparator} any vertex that is not contained in a vertex-edge separator 
	in a graph $G$ is also not contained in a vertex-edge separator in $G_1 \cup G_2$ with
	$G=G_1\veident G_2$. As every vertex in $G^\prime$ is at some point contained in the set $S$
	and, when regarded, is only kept in the graph if it is not contained in a vertex-edge separator, no vertex of $G^\prime$ 
	can be contained in a vertex-edge separator. This proves the correctness of the Algorithm.
	
	Now we discuss the running time of the algorithm.
	If $G$ contains only one vertex, the algorithm terminates after the first iteration, as no vertex-edge separation step can be applied.
	Now assume that~$n\geq 2$.
	Algorithm~\ref{alg:components} never creates a component with only one vertex.
	Thereby any component of $G^\prime$ contains at least two vertices. Let $k$ be the 
	number of vertex-edge separation steps taken during the whole procedure. 
	We get that the number of components in $G^\prime$ is exactly $k+1$, so the 
	number of vertices in~$G^\prime$ is at least $2\cdot(k+1)$. As in each iteration exactly one additional 
	vertex is added to the graph, we know that $|V(G^\prime)|=k+n$. Thus, $k+n\geq 2\cdot(k+1)$ which implies~$k\leq n-2$.
	As already mentioned, we can
	find a cut-edge and split the graph using a slightly altered 
	version of the usual \emph{lowpoint algorithm} for finding biconnected components, cf.~\cite{Hopcroft1973}.
	This algorithm can be implemented to run in time $\mathcal{O}(n+m)$. 
	Thus any call to \textsc{TestAndDecompose} needs at most time $\mathcal{O}(n+m)$.
	Altogether Algorithm~\ref{alg:components}
	can be implemented to run in time~$\mathcal{O}(n(n+m))$.
\end{proof}

Next we want to use Algorithm~\ref{alg:components} to find out 
if the cycle number of a biconnected graph $G$ is unique. As pointed out 
earlier, we will do this by simply testing if all components remaining,
after Algorithm~\ref{alg:components} has terminated, are Eulerian 
multiedges.
\begin{algorithm}
	\begin{algorithmic}[1]{isCycleNumberUnique$(G)$}
		\INPUT{Biconnected graph~$G$.}
		\OUTPUT{$\begin{cases}
			\texttt{True},&\text{ if cycle number is unique,}\\
			\texttt{False},& \text{else.}
			\end{cases}$}
		\STATE $H=\textsc{VE-Components}(G)$
		\FORALL{$v\in V(H)$}
		\IF{$|N(v)|\neq 1$ or $\#$ of incident edges is odd}
		\STATE\RETURN\texttt{False}
		\ENDIF
		\ENDFOR
		\STATE \RETURN\texttt{True}
	\end{algorithmic}
	\caption{Test if cycle number of a graph is unique}\label{alg:cycleunique} 
\end{algorithm}

\begin{theorem}
	\label{thm:cycleunique}
	Algorithm~\ref{alg:cycleunique} correctly decides if the cycle number of 
	a biconnected graph $G$ with $n$ vertices and $m$ edges is unique. It can be implemented to run in 
	time $\mathcal{O}(n(m+n))$.
\end{theorem}
\begin{proof}
	First note, that a graph $H$ is a disjoint union of Eulerian multiedges
	if and only if each vertex has exactly one neighbour and the number of incident
	edges is even. This proves that Algorithm~\ref{alg:cycleunique} returns
	\texttt{True} if and only if the graph $H$ in the algorithm is a collection
	of Eulerian multiedges. The running time is clear as the whole algorithm is
	clearly dominated by the running time of Algorithm~\ref{alg:components}. 
	
	It remains to show that this indeed is a necessary and sufficient
	condition for the graph $G$ to have a unique cycle number. It is easy to
	see that $G$ is contained in $\mathcal{G}$, as in order to create it we 
	only have to do the algorithm backwards. Now assume that $G$ has a unique
	cycle number but Algorithm~\ref{alg:components} does not terminate with
	a collection of Eulerian multiedges. Then there exists a component $H$, which
	is not an Eulerian multiedge and
	does not allow a vertex-edge separation step.
	This implies that $H\notin\mathcal{G}$ and by Theorem~\ref{thm: main}
	we have $\nu(H)>c(H)$. If we now apply Lemma~\ref{lem: operations preserve cycles}
	multiple times, we get that $\nu(G)>c(G)$, which is a contradiction 
	to $G$ having unique cycle number.
\end{proof}

Observe that we can also use Algorithm~\ref{alg:components} to reduce computation
of minimum or maximum cycle number to smaller graphs: 
We simply run the algorithm on a given graph $G$ and compute a minimum respectively maximum cycle
decomposition in the outputted components. By Lemma~\ref{lem: operations preserve cycles} 
we can then puzzle the cycle 
decomposition together in order to obtain a minimum or maximum cycle
decomposition of $G$.

\section*{Acknowledgement} 
We thank Till Heller, Sebastian Johann, Sven O.\ Krumke and Eva Schmidt for their 
helpful comments on this article.

\bibliography{references}
\bibliographystyle{alpha}

\end{document}